\documentclass[a4paper,10pt]{article}
\usepackage[envcountsect]{beamerarticle}

\ifx\pdfpageheight\undefined
   \usepackage[dvips,colorlinks=true,linkcolor=blue,citecolor=red,%
      urlcolor=green]{hyperref}
   \usepackage[dvips]{graphicx}
   \makeatletter
   \edef\Gin@extensions{\Gin@extensions,.mps}
   \makeatother
\else
   \usepackage[pdftex]{graphicx}
   \usepackage[bookmarksopen=false,pdftex=true,breaklinks=true,%
      backref=page,pagebackref=true,plainpages=false,%
      hyperindex=true,pdfstartview=FitH,colorlinks=true,%
      pdfpagelabels=true,colorlinks=true,linkcolor=blue,%
      citecolor=red,urlcolor=blue,hypertexnames=false%
      ]%
   {hyperref}
\fi

\usepackage{amsmath}
\usepackage{amssymb}
\usepackage{graphicx}
\usepackage[utf8]{inputenc} 
\usepackage[T1]{fontenc}
\usepackage{textcomp}

\usepackage[scaled=0.92]{helvet}
\usepackage{latexsym}
\usepackage{stmaryrd}

\usepackage{epsfig,bbm,multirow,psfrag,pb-diagram,amsmath,amssymb,enumerate}

\newcommand\ZZ{\mathbb{Z}}
\newcommand\NN{\mathbb{N}}

\newcommand\RR{\mathbb{R}}
\newcommand \gR{\mathbf{R}} 
\newcommand\CC{\mathbb{C}}

\newtheorem{remark}{Remark}[section]

\newcommand \formule[1]{{\left\{ {\arraycolsep2pt\begin{array}{lll} #1 \end{array}}\right.}}

\newcommand{\gui}[1]{``{#1}''}

\newcommand\fm{\mathfrak{m}}

\newcommand \Xn {X_1,\ldots,X_n}
\newcommand \fn {f_1,\ldots,f_n}
\newcommand \xn {x_1,\ldots,x_n}

\newcommand \kuX {\gk[\uX]}

\newcommand \uX {\underline{X}}
\newcommand \uY {\underline{Y}}

\newcommand \ux {\underline{x}}
\newcommand \ua {\underline{a}}

\newcommand \uwha {\underline{\wh{a}}}
\newcommand \wha {{\wh{a}}}
\newcommand \uxi{{\underline{\xi}}}
\newcommand \uze{{\underline{0}}}
\newcommand \uwv {\underline {\wi v}}

\newcommand \AXn {\gA[\Xn]}

\newcommand \Kxn {\gK[\xn]}
\newcommand \kXn {\gk[\Xn]}
\newcommand \Axn {\gA[\xn]}

\newcommand \AuX {\gA[\uX]}
\newcommand \Aux {\gA[\ux]}
\newcommand \Co{\mathbb{C}}
\newcommand \gA{\mathbf{A}}
\newcommand \gB{\mathbf{B}}
\newcommand \gC{\mathbf{C}}
\newcommand \gD{\mathbf{D}}
\newcommand \gE{\mathbf{E}}
\newcommand \gK{\mathbf{K}}
\newcommand \gk{\mathbf{k}}

\newcommand \pcB {{\partial\cB}}
\newcommand \cB {\mathcal{B}}

\newcommand{\A}{{\gA}}

\newcommand{\D}{{\gD}}
\newcommand{\K}{{\gK}}

\newcommand\junk[1]{}

\newcommand\abs[1]{\left|{#1}\right|}

\newcommand\gen[1]{\left\langle{#1}\right\rangle} 
\newcommand\geN[1]{\big\langle{#1}\big\rangle} 
\newcommand\sotq[2]{\so{\,#1\mid#2\,}} 
\newcommand\so[1]{\left\{{#1}\right\}} 
\newcommand \ov[1] {\overline{#1}}
\newcommand \wh[1] {\widehat{#1} }
\newcommand \wi[1] {\widetilde{#1} }

\newcommand\lrb[1] {\llbracket #1 \rrbracket}
\newcommand\lrbn {\lrb{1..n}}

\newcommand\lrbs {\lrb{1..s}}

\newcommand \som {\sum\nolimits}

\newcommand \Span {\mathrm{Span}}

\renewcommand \mod {\;\mathrm{mod}\;}
\newcommand \Alg {$\gA$-algebra }
\newcommand \Algs {$\gA$-algebras }
\newcommand \Algz {$\gA$-algebra}

\newcommand \Klg {$\gK$-algebra }

\newcommand \Klgsz {$\gK$-algebras}
\newcommand \Kev {$\gK$-vector space }
\newcommand \kev {$\gk$-vector space }

\newcommand \klg {$\gk$-algebra }

\newcommand \klgz {$\gk$-algebra}

\newcommand \Amo {$\gA$-module }

\newcommand \Amoz {$\gA$-module}

\newcommand \coes {coefficients }

\newcommand \Tho {Theorem }
 
\newenvironment{pf}[1]{
\trivlist \item[\hskip \labelsep{{\it #1}.}]}{\hfill\mbox{$\Box$}
\endtrivlist}
\makeatletter

\begin{document}

\author{M. Emilia Alonso\thanks{Universidad Complutense, Madrid, Espa\~na. {\tt mariemi@mat.ucm.es}.  \url{http://www.mat.ucm.es/imi/People/Alonso_Garcia_MariaEmilia.htm} Supported by Spanish GR MTM-2011-22435 and MTM-2014-55565.}
\and
Henri Lombardi\thanks{Univ. de Franche-Comt\'e, 25030
Besan\c{c}on cedex, France. {\tt henri.lombardi@univ-fcomte.fr}. \url{http://hlombardi.free.fr/} Supported by Spanish GR MTM-2011-22435 and MTM-2014-55565.} }

\title{Local B\'ezout Theorem for Henselian rings}

\maketitle

\begin{abstract}In this paper we prove what we  call {\it Local B\'ezout Theorem} 
(Theorem \ref{bezloc}). It is a formal abstract algebraic version, in the frame of Henselian rings and $\fm$-adic topology, 
of a well known theorem in the analytic
complex case. This classical theorem says that, given an isolated  point of multiplicity $r$ 
as a zero of a local complete intersection,  after deforming the \coes of these equations we find 
in a sufficiently small  neighborhood of this point exactly $r$ isolated zeroes  counted with 
 multiplicities.  
Our main tools are, first the {\it border bases} \cite{Mou99},  which turned out to be an 
efficient computational tool to deal with  deformations of algebras.
Second we use an important result of de Smit and Lenstra 
\cite{SmLe97}, for which there exists a constructive proof in \cite{lq15}. Using these  tools we find a very simple
proof of our theorem, which seems new in the classical literature.

 \end{abstract}

{\bf Keywords.} 
Local B\'ezout Theorem, Henselian rings, Roots continuity, Stable computations, Constructive Algebra. 

{\bf MSC 2010.} primary 13J15; secondary 13P10, 13P15, 14Q20, 03F65

\section{Introduction}

In this paper we use ideas from computer algebra  to prove what we  call {\it Local B\'ezout Theorem} 
(Theorem \ref{bezloc}). It is a formal abstract algebraic version, in the frame of Henselian rings and $\fm$-adic topology, 
of a well known theorem in the analytic
complex case.  This classical theorem says that, given an isolated  point of multiplicity $r$ 
as a zero of a local complete intersection,  after deforming the \coes of these equations we find 
in a sufficiently small  neighborhood of this point exactly $r$ isolated zeroes  counted with 
 multiplicities. 

As far as we know the proofs of this classical  result in the literature are essentially:
 by Arnold using powerfully 
 the topological degree and Weierstrass theorems  \cite[section I-4.3]{ArBook},  
and another one that can be deduced from  Griffiths-Harris~\cite[Residue Theorem, p.\ 656]{GH} using the theory of residues and its relationship
 with multiplicity. 
 
Here we state and prove an {\em algebraic version}
of this theorem in the setting of arbitrary Henselian rings and $\fm$-adic topology. 
We are somehow inspired by Arnold in \cite[section I-4.3]{ArBook}, exploiting an abstract version of Weierstrass division
(in a Henselian ring) and introducing also an  
abstract version of which he called  the \gui{multilocal ring}. Roughly speaking we consider a 
finitely presented \Algz, where $(\gA, \fm, \gk)$ is a local Henselian ring such that the special point is a \klg with an isolated zero of multiplicity $r$ and we prove that the \gui{multilocal ring} 
determined by this point is a free \Amo of rank $r$.

\smallskip Our main tools are, first the {\it border bases} \cite{Mou99},  which turned out to be an 
efficient computational tool to deal with  deformations of algebras.
Second we use an important result of de Smit and Lenstra 
\cite{SmLe97}, for which there exists a constructive proof in \cite{lq15}.

In Section~\ref{subsecBB} we recall the definition and  main properties of border bases. 

 We point out that, 
 to obtain a fully algorithmic proof of our results, we rely on the
 constructive proof 
of the Multivariate Hensel Lemma (MHL for short) given in \cite{ACL2014}. 
Also, in order to get true algorithms, fields are assumed to be discrete and local rings to 
be residually discrete.

The pure abstract algebraic form of Local B\'ezout Theorem  is given in Theorems  \ref{bezloc1} and~\ref{bezloc}.

\smallskip  Going futher into details let us explain the  form of Local B\'ezout Theorem
we are interested~in. 

Assume $(\gA, \fm,\gk )$ is a local normal domain, $f_1,\ldots,f_n\in\AXn$ and let
$$
\gB:=\AXn/\!\gen{f_1,\ldots,f_n}\hbox{ and }\gC=\gB_{\fm +\gen{\ux}}=
\A[x_1,\ldots,x_n]_{ \fm +\gen{\ux}},
$$
where  $\gen{\ux}=\gen{\xn} $, $x_i $ is $X_i \mod  \gen{f_1,\ldots,f_n}  $, 
and $ f_i(\uze)\in \fm$ \hbox{for $i=1,\ldots,n$}.
We denote by~$\K$ the quotient field of $\A$.  We   assume   $\K$  to be an  algebraically closed field 
and  therefore   $\A$  to be a Henselian ring.

We assume that the \klg 
$$
\overline{\gC}:=(\kXn/\!\gen{\ov{f_1},\ldots,\ov{f_n}})_{ \gen{\ov{x_1},\ldots,\ov{x_n}} }=\gk[\ov{x_1},\ldots,\ov{x_n}]_{ \gen{\ov{x_1},\ldots,\ov{x_n}} }
$$ 
is zero-dimensional, where $\ov{f_i}$ (resp. $\ov{x_i}$) is the image of $f_i$ (resp. ${x_i}$) by $\otimes_\gA\gk$. 

Since  $\K$ is algebraically closed it is plausible 
to speak about the continuity of the roots. 

The algebraic form of Local B\'ezout Theorem (Theorem \ref{bezloc}) says that there are finitely many zeroes of $(f_1,\ldots,f_n)$
 above the residual zero $(0,\ldots,0)$ (i.e., with coordinates 
 in $\fm$), and the sum of their multiplicities
 equals the dimension $r$ of $\overline{\gC}$ as $\gk$-vector space,
i.e., the multiplicity of the residual zero. 
This implies also that the \Klg $\gC\otimes_\gA\gK$ is finite free of rank $r$.


\smallskip As application of the precedent sections,  in section~\ref{classical},  we obtain, from the abstract theorem, the classical one. Here the characterization of  border bases introduced by Bernard Mourrain in \cite{Mou99} reveals to be a crucial tool. 

Our proof shows, (very much in the spirit of \cite[Eisermann]{Eis2012})  that this classical result is also valid for the case of the field $\gR(i)$ (instead of ${\mathbb C}$), where $\gR$ is any real closed field and we consider in $\gR(i)\cong \gR^2$ the euclidean topology.  


\section{Useful tools}\label{secUseful}

\subsection{Theorem of de Smit and Lenstra}
\begin{theorem} \label{thdeSmitLenstra} 
 \emph{(\Tho of de Smit {\&} Lenstra, flatness, \cite{SmLe97})}
\\
Let $\gA$ be an arbitrary commutative ring. If an \Alg 
$$\gB=\AXn/\!\gen{f_1,\dots,f_n}$$ 
is finite, then it is flat (so, it is free if $\gA$ is local).
\end{theorem}
%

A constructive proof due to Claude Quitté  is explained in \cite{lq15}.

\subsection{Border bases}
\label{subsecBB} 

In this subsection, $\gA$ is an arbitrary commutative ring.

\medskip   \noindent {\bf (2.2.1)} 
In the sequel we shall identify the semi-groups  $X_1^\NN\cdots X_n^\NN$  and $\NN^n$.
  A nonempty finite subset $\cB\subset \NN^n$ is called 
 \emph{closed by division} if 
 for every $X^{\gamma}$, $X^{\gamma'}$, if $X^{\gamma} \in \mathcal{B}$ and
 $X^{\gamma'} \vert\; X^{\gamma}$ then $X^{\gamma'} \in \mathcal{B}$.

\noindent  In the sequel any finite subset $\NN^n$ denoted by $\cB$ will be assumed nonempty and closed by division.

\noindent  We call {\it border of} $\mathcal{B}$ the following finite subset of 
 $\NN^n$,
 $$\pcB:=
 (X_1\cB \,\cup\, \cdots \,\cup\, X_n \cB) \setminus \cB $$

\medskip    \noindent {\bf (2.2.2)} Let $\gA$ be  a  ring, $I$ a finitely generated ideal of $\AXn=\AuX$ and 
$$\gB:=\AuX/I=\gA'[\xn]
$$  
($x_i$  the image of $X_i$ in $\gB$ and $\gA'$ the image of $\gA$). 
 Given a finite subset $\cB \subset \NN^n$ as above,
 we call \emph{rewriting rules for $\gB$ w.r.t.\ $\cB$} a set of equalities in $\gB$
 as follows,
\begin{equation} \label {eqRewritingRule}
\so{\,x^{\beta}=\som_{\alpha\in \cB} h_{\beta,\alpha}x^{\alpha}:
 \beta \in \pcB\,} \hbox{ where }h_{\beta,\alpha}\in \gA
\end{equation}
Formally, we define the rewriting rules as being the polynomials $h_\beta(\uX)=X^\beta-\som_{\alpha\in \cB} h_{\beta,\alpha}x^{\alpha}$.
The equalities $(\ref{eqRewritingRule})$ mean precisely that the $h_\beta(\uX)$'s belong to the ideal $I$.
 
\noindent  If $\gB$ defined as above has  rewriting rules w.r.t.\ to $\cB$, then 
 the set~$\{\,x^{\alpha}: \alpha\in \cB  \,\} $ 
 generate $\gB$ as \Amoz. In  particular $\gB$  is a finite \Amoz.

\medskip    \noindent {\bf (2.2.3)} We call $(\cB,(h_\beta)_{\beta\in \pcB})$ a {\it border basis}
 of $\gB/\gA$ if 
\begin{itemize}
\item the $h_\beta(\uX)$'s are  rewriting rules in $\gB$ w.r.t. $\cB$ and
\item  $(x^\alpha)_{\alpha\in \cB}$ is a basis of $\gB$ as an \Amoz.
\end{itemize}
This implies in particular that $\gA'=\gA$.

\noindent Given any morphism $\rho:\gA\to\gC$ of  rings a border basis
$(\cB,(h_\beta)_{\beta\in \pcB})$ of $\gB/\gA$ is transformed by~$\rho$ in a border basis
$(\cB,(\rho(h_\beta))_{\beta\in \pcB})$ of $(\gB\otimes_\gA\!\gC)/\gC$.  
 
\medskip    \noindent {\bf (2.2.4)} Given  $\gA$  
 and $\cB$ as above, assume be given a set of ``abstract'' rewriting rules (\ref{eqRewritingRule}) (without $\gB$). 
Define  
$I:=\gen {X^\beta-\sum_{\alpha \in \cB} h_{\beta,\alpha}X^{\alpha}:  
\beta \in \partial{\cB}}$ and $\gB:=\AuX/I$. 
We can ask 
wether  $\gB$ is a free \Amo of basis $\so{x^{\alpha}: \alpha\in \cB}$.
I.e., do we get a border basis of $\gB$? For this,  let us denote by 
$\Span_{\gA}( \cB)$ the free \Amo generated by the terms of  $\cB$, and  for every $i$
$$
\mu_{X_i}: \Span_{\gA}( \cB) \rightarrow \Span_{\gA}( \cB) 
$$ 
the linear map given by
$$
\mu_{X_i}(X^{\alpha })=
\formule{ 
X_iX^{\alpha } \qquad \quad \quad \hbox{ if } X_iX^{\alpha}\in \cB,
\\[.3em]
\sum_{\alpha \in \cB} h_{\beta,\alpha}X^{\alpha} \quad \hbox{ if }
X_iX^\alpha=X^\beta\in \pcB. 
}
$$
 Now let  $\Lambda_i$ denote
 the matrix of $\mu_{X_i}$ 
 w.r.t. the basis $ \cB $ of the free \Amo $\Span_{\gA}( \cB)$, then, see \cite{Bra} and \cite{Mou99}

\begin{theorem} \label{thBBR}
$\gB$ is a  free \Amo with basis 
$\so{x^{\alpha}: \alpha\in \cB}$  iff
$$
\Lambda_i\Lambda_j=\Lambda_j\Lambda_i \;\hbox{ for }i,j=1,\dots,n.
$$
\end{theorem}

\medskip   \noindent {\bf (2.2.5)} In the particular case in which $\gA$ is a discrete field $\gk$ and $\gB$ is an Artinian 
\klgz,  the Gröbner basis algorithm provides a border basis of $\gB$. Indeed, it is well known that, w.r.t.\ to an admissible
monomial ordering, the monomials ``under the staircase'' form a $\gk$-basis $\cB$ of $\gB$, and the rewriting rules
are given by computing the remainder of the division of $X^\beta$ by the Gröbner basis (for~$\beta\in\pcB$).
 Mourrain in \cite{Mou99} introduced an algorithm to compute a border basis of an Artinian  \klg 
without using the theory of  Gröbner basis and requiring for $\cB$ a weaker property 
  than to be closed by division.


\medskip    \noindent {\bf (2.2.6)}  One  important fact  with border bases  is that they are more suitable  
 for computational purposes specially when data is given by approximate values.  
 Another one is that they can be useful in  more general cases when Gröbner bases over rings are not available or not easy to manage.

\section{The Bézout local theorem}

Let us recall two theorems in \cite{AL2010}.

\begin{theorem}\label{Mather} \emph{(\cite[Theorem 1]{AL2010})}
 Let $(\A,\fm,\gk)$ be a  local Henselian ring.
 \\
 Let $f_1, \ldots, f_m \in \AXn=\AuX$,   $\ov{f_1},\ldots,\ov{f_m}$ their images in $\kuX$,  
$$
\gB:=\AuX/\!\gen{f_1, \ldots, f_m}=\Axn
$$ 

\vspace{-1em}
\noindent and 
$$
\ov{\gB}:=\gB/\fm\gB=\kXn/ \!\gen{\ov{f_1},\ldots,\ov{f_m}}=\gk[\ov{x_1},\dots,\ov{x_n}].
$$ 
Let $\gB_1:=
\gB_{\fm+\gen{\ux}}$, $\ov{\gB_1}:=\ov\gB_{\gen{\ov{x_1},\dots,\ov{x_n}}}$ and assume $1 \leq \dim_\gk(\ov{\gB_1})<\infty$.  
\\
Then  the \Alg
$\gB_1$ is a finitely generated \Amoz.
\end{theorem}

\begin{theorem}\label{bezloc0} \emph{(\cite[Theorem 5]{AL2010})}
 Let $(\A,\fm,\gk)$ be a  Henselian DVR.\\
 Let $f_1, \ldots, f_n \in \AXn=\AuX$,   $\ov{f_1},\ldots,\ov{f_n}$ their images in $\kuX$,  
$$\gB:=\AuX/\!\gen{f_1, \ldots, f_n}=\Axn$$ 

\vspace{-1em}
\noindent and 
$$
\ov{\gB}:=\gB/\fm\gB=\kXn/ \!\gen{\ov{f_1},\ldots,\ov{f_n}}=\gk[\ov{x_1},\dots,\ov{x_n}].
$$ 
Let $\gB_1:=\gB_{\fm+\gen{\ux}}$, $\ov{\gB_1}:=\ov\gB_{\gen{\ov{x_1},\dots,\ov{x_n}}}$ and assume $\dim_\gk(\ov{\gB_1})=r\geq 1$.  
\\
Then  the \Alg
$\gB_1$
 is a free \Amo of rank $r$, whose basis is given by lifting any $\gk$-basis of~$\ov \gB_1 $.
\end{theorem}

We give first a result that implies a generalization and an important precision in Theorems  \ref{Mather} and \ref{bezloc0} in case of a residually finite global complete intersection ($m=n$ and the residual algebra is finite): see corollary \ref{superbezloc0}.

In the sequel we use the notation $\fm[\uX]$ for the ideal of $\AuX$
generated by $\fm$ (the polynomials with coefficients in $\fm$) and $\fm+\gen{\uX}$ for the ideal $\fm[\uX]+\gen{\uX}$
(the polynomials $\in\AuX$ with constant coefficient in $\fm$).

\begin{theorem}\label{bezloc1}
 Let $(\gA,\fm,\gk)$ be a local Henselian ring, 
$\gB:= \Axn$ be a finitely generated \Alg  and set     
$\ov{\gB}:=\gB/\fm\gB=\gB\otimes_A \gk $  the residual algebra.
Assume that \smash{$\ov{\gB}$} is a nonzero finitely generated $\gk$-vector space. Then it holds:

\begin{enumerate}
 \item \noindent There exists $s \in 1+\fm \gB$ s.t.  the \Alg 
$\gB[1/s]$ is a finitely generated  \Amoz. More precisely, in case we know a border basis 
of $ \ov{\gB}$ as $\gk$-algebra, 
its rewriting rules can be lifted,
 to provide, for a suitable $ s \in 1+\fm \gB$, a system of generators of  $\gB[1/s]$ as an $\gA$-module. 

\item\noindent If in addition the \Alg $\gB$ is presented presented with an equal number of generators and 
relations, that is 
  $\gB:=\gA[X_1, \ldots,X_n]/\!\gen{f_1, \ldots, f_n}=\Axn$, 
and if \smash{$\dim_\gk(\ov{\gB})=r\geq 1$}, there exists $S\in 1+\fm[\uX]$ s.t.\ letting $s=S(\ux)$,  $\gB[1/s]$ is a free  \Amo  of rank $r$ 
 whose basis is given by lifting any $\gk$-basis of $\ov\gB$.
 In this case  any border basis of the residual algebra  $\ov\gB=\ov{\gB[1/s]}$ can be lifted
with its rewriting rules to a border basis of $\gB[1/s]$ with its
rewriting rules. 
\end{enumerate}
\end{theorem}

\noindent \emph{Note}.  Note that in order to be able to compute a $\gk$-basis of $\ov\gB$, we need a priori that $\gB$ be finitely presented as an \Algz. So, in item \emph{1} we do not assume that $\ov\gB$ has a $\gk$-basis).

\smallskip An intuitive meaning of the first part of the statement of the theorem is that, when $\gA$ is a domain with quotient field $\gK$ and the polynomial system is 
residually zero-dimensional, inverting $s$ ``maps at infinity'' all zeroes of $I$ in $\gK$  which are not integral
over $\gA$. We explain this intuition through  Corollary \ref{corbezloc1}. 

\begin{corollary} \label{corbezloc1}
Same hypotheses and notations as in Theorem \ref{bezloc1} 1.
Assume moreover that $\gB:= \AXn/\!\gen{f_1,\dots,f_m}=\Axn$ and 
we  are  given an 
\Alg $\rho:\gA\to\gE$ and a zero $(\xi_1,\dots,\xi_n,\zeta)$ of $\{f_1,\dots,f_m,ZS(\uX)-1\}$ in $\gE$.
Then each $\xi_i$ is integral over $\rho(\gA)$.
\end{corollary}
\begin{proof}
Since $\gB[1/s]$ is a finite \Amoz, each $x_i$ is integral over $\gA$ in $\gB[1/s]$: in fact 
$x_i$ annihilates the characteristic polynomial of the matrix of multiplication by $x_i$ 
with respect to  a system of generators 
of $\gB[1/s]$ as \Amoz
. The morphism $\rho$ gives by factorization a morphism of \Algs $\gB[1/s]\to\gE$ mapping
$(x_1,\dots,x_n,z)$ to $(\xi_1,\dots,\xi_n,\zeta)$. So each $\xi_i$ annihilates a monic polynomial with coefficients in $\rho(\gA)$. 
\end{proof}
%

\begin{corollary} \label{superbezloc0}
Same hypotheses and notations as in Theorem \ref{bezloc1} {1}, and in addition assume  that
the  \klg $\ov\gB$ is local, more precisely that  $\ov{\gB}=\ov\gB_{\gen{\ov{x_1},\dots,\ov{x_n}}}$ 
(this means  that each $\ov{x_i}$ is nilpotent, 
or also that $(0,\dots,0)$ is the unique zero of $\so{\ov{f_1},\dots,\ov{f_m}}$ in an algebraic closure of $\gk$). 
\begin{enumerate}
\item There exists an $s\in1+\fm\gB$
 such that $\gB[1/s]=\gB_{\fm+\gen{\ux}}$ 
 is a finitely generated \Amoz.
\item  If in addition $\gB:=\gA[X_1, \ldots,X_n]/\!\gen{f_1, \ldots, f_n}=\Axn$, then   the \Alg $\gB_{\fm+\gen{\ux}}$ is a free \Amo of rank $r$, whose basis 
is given by lifting any $\gk$-basis of~$\ov {\gB}$ with its rewriting rules.
\end{enumerate}
\end{corollary}
\begin{proof} \emph{1}. First we note that $\ov{\gB_{\fm+\gen{\ux}}}=\ov\gB_{\gen{\ov{x_1},\dots,\ov{x_n}}}=\ov\gB$.
We apply Theorem \ref{bezloc1}. We get an $s\in1+\fm\gB$ such that $\gB[1/s]$ is a finitely generated  \Amoz. 
Since $s$ is invertible in $\gB_{\fm+\gen{\ux}}$, we get an $\gA$-morphism $\varphi:\gB[1/s]\to \gB_{\fm+\gen{\ux}}$. We show 
\gui{$\varphi$ is a canonical isomorphism}: the two localizations of $\gB$ are the same one. This means that 
any element $c\in1+\fm\gB+\gen{\ux}$ is invertible in $\gB[1/s]$.
Since $\ov c=1$ in $\ov\gB_{\gen{\ov{x_1},\dots,\ov{x_n}}}$ and $\ov\gB=\ov\gB_{\gen{\ov{x_1},\dots,\ov{x_n}}}$, we get $\ov c=1$ 
in $\ov{\gB[1/s]}$
and the result follows from Nakayama: the multiplication by $c$ in $\gB[1/s]$ is an $\gA$-endomorphism
 which is residually onto, hence  itself is onto and $c$ is invertible in $\gB[1/s]$.

\smallskip \noindent  \emph{2}.
We are in the hypothesis of \ref{bezloc1} \emph{2} and $\gB[1/s]=\gB_{\fm+\gen{\ux}}$.
\end{proof}
%

\begin{corollary} \label{cor2bezloc1}
Same hypotheses and notation as in Corollary \ref{corbezloc1}. Assume moreover that 
$\ov{\gB}=\ov\gB_{\gen{\ov{x_1},\dots,\ov{x_n}}}$ (as in \ref{superbezloc0}~1). Then each $\xi_i$ is ``integral over $\fm$'' (see the precise meaning in the proof).
\end{corollary}
\begin{proof} Since $\ov{x_i}$ is nilpotent, the multiplication $\mu_{\ov{x_i}}$ by $\ov{x_i}$ in $\ov\gB$ is given by a lower triangular matrix $M_i$ with zeroes on the diagonal (after a suitable change of basis as a $\gk$-vector space).
Since we have a generator system $B$ of $\gB[1/s]$ as an \Amo which is lifted from a border basis $\cB$ of $\ov\gB$, and since the rewriting rules are also lifted, the multiplication by $x_i$ in $\gB[1/s]$ can be expressed w.r.t. a generator system $B'$ (obtained from $B$ by lifting the above change of vector space basis) by a lifting of the above matrix $M_i$ and  the characteristic polynomial of this lifted matrix has the form 
\smash{$T^r+\sum_{j=0}^{r-1}\mu_{i,j}T^{r-j}$} with $\mu_{i,j}\in\fm^j$.  \\So,
we get an equality in $\gE$: $\xi_i^r+\sum_{j=0}^{r-1}\rho(\mu_{i,j})\xi_i^{r-j}=0$ and $\mu_{i,j}\in\fm^j$.
\end{proof}
%


\begin{pf}{Proof of Theorem \ref{bezloc1}} ~ \\
\emph{1}. We set   $\gB=\AXn/I$ for some ideal $I\subseteq \AuX$. 
As $\ov{\gB}$  is a finitely generated nonzero $\gk$-vector space, we can represent it as a quotient of some finitely presented nonzero \klg 
$$\gk[X_1, \ldots,X_n]/\! \gen {\ov{f_1}, \ldots, \ov{f_m}}$$
for some $f_i\in I$. So $\gB$ is a quotient of the finitely presented
\Alg $${\gB_1}= \gk[X_1, \ldots,X_n]/\! \gen {{f_1}, \ldots, {f_m}}$$
whose residual \klg $\ov{\gB_1}$ is a $\gk$-vector space with a
finite $\gk$-basis.\\
Hence w.l.o.g.\ we can assume that $\gB$ is finitely presented.  
Hence $\ov\gB$ has a border basis. There exists a finite set of monomials $\cB \subset \NN^n$ containing $1$ and closed by division, 
such that
$\{x^{\alpha}: \alpha \in \pcB \}$ is a $\gk$-basis of $\ov{\gB}$, and 
for $\beta \in \pcB$
$$
h_{\beta}^{0}(X):= X^{\beta}-\som_{\alpha \in \cB } u_{\beta,\alpha}^{0} X^{\alpha} \quad \quad 
(\hbox{with }u_{\beta,\alpha}^{0}  \in \gk)
$$
are  the corresponding rewriting rules. 
%


\medskip 
Now we follow  the constructions in the proof of \cite[Theorem 1]{AL2010} (see Claim 4 in \cite{AL2010}). 
As $\gen {\ov{f_1}, \ldots, \ov{f_m}} =\gen {h^0_{\beta}: \beta\in \pcB } \subseteq \kXn$, 
we have  $h^0_{\beta}=\sum_{i=1}^m \ov{p_{i,\beta}}\;\ov{f_i}$, with some ${ p_{i,\beta}} \in \AuX$. 
For $\beta \in \pcB$ we put
$$
H_{\beta}(\uX):=\som_{i=1}^m p_{i,\beta} f_i \quad \quad (\hbox{so } \ov{H_{\beta}}=h_{\beta}^0).
$$ 

Next we reduce these polynomials $H_{\beta}$ with the following 
``formal rewriting rules''
$$
\wi{h_{\beta}}:= X^{\beta}-\som_{\alpha \in \cB} \wi{u_{\beta,\alpha}}X^{\alpha},
$$
where $\wi{u_{\beta,\alpha}}$
 are variables, whose suitable values in $\A$ we are looking for.

Following \cite{AL2010} and \cite{ABM} we get

\begin{equation} \label {eqHbetatilde}
H_{\beta}= \som_{\beta' \in \pcB} \wi{Q_{\beta,\beta'}} \wi{h_{\beta'}} + \wi{R_{\beta}}
\end{equation}
for  polynomials $\wi{Q_{\beta,\beta'}}((\wi{u_{\beta,\alpha}}),\uX) \in \gA[(\wi{u_{\beta,\alpha}})_{\beta\in  \pcB,\alpha\in \cB},\uX]$
and 
$$\wi{R_{\beta}} =\sum\nolimits_{\alpha\in\cB} \wi{R_{\beta, \alpha}} X^{\alpha}\quad \quad \hbox{with}
\,\, \wi{R_{\beta,\alpha}}((\wi{u_{\beta,\alpha}}))\in  \gA[(\wi{u_{\beta,\alpha}})].
$$

Moreover 
$$\wi{\Delta}((\wi{u_{\beta,\alpha}}),\uX):= \det((\wi{Q_{\beta,\beta'}})_{\beta,\beta'\in\pcB}) \in 1+\fm[ (\wi{u_{\beta,\alpha}}),\uX].$$

In \cite{AL2010} and \cite{ABM} we proved that $(u_{\beta,\alpha}^{0})$ is  an isolated simple zero of the system $\so{ \ov{\wi{R_{\beta, \alpha}}} =0}$. Since $\A$ is Henselian, by MHL, there exists  a unique solution 
$(u_{\beta,\alpha}) \in \gA^{\pcB\times \cB}$  of the system $ \so{\wi{R_{\beta, \alpha}} =0
}
$
lifting the solution $(u_{\beta,\alpha}^0) \in \gk^{\pcB\times \cB}$.

For $\beta \in \pcB$, we define 
\[ 
\begin{array}{rcl} 
 h_{\beta} & :=  & X^{\beta}-\sum_{\alpha\in\cB} u_{\beta,\alpha} X^{\alpha},  \\[.3em] 
Q_{\beta,\beta'}(\uX)  & :=  &  \wi{Q_{\beta,\beta'}}(({u_{\beta,\alpha}}),\uX) \in \AuX, \\[.3em] 
S(\uX)  & :=  &  \wi{\Delta}((u_{\beta,\alpha}),\uX) =\det((Q_{\beta,\beta'})_{\beta,\beta'\in\pcB}) \in 1+\fm[\uX].
 \end{array}
\]
Equalities (\ref{eqHbetatilde}) give
\begin{equation} \label {eqHbeta}
H_{\beta}= \som_{\beta' \in \pcB} {Q_{\beta,\beta'}} {h_{\beta'}}.
\end{equation}

\smallskip Let $s=S(\ux)\in\Aux=\gB$.
The ideal $\gen{f_1,\dots,f_m,SZ-1}\subseteq  I \subseteq  \gA[\uX,Z]$ contains polynomials~$h_\beta$'s for $\beta\in\cB$
since $h_\beta\equiv ZSh_\beta$ and $Sh_\beta$ is expressed from the $H_{\beta'}$'s using the cotransposed matrix 
of~$(Q_{\beta,\beta'})_{\beta,\beta'\in\pcB}$. As $f_i $ belong to the ideal $ I$, we get a well defined \Alg morphism   
\[ 
\begin{array}{c} 
 \Phi: \gC=\AuX/\!\gen {(h_{\beta})_{\beta\in\pcB}}
\longrightarrow \A[\uX, Z]/\gen {I,SZ -1} =\gB[1/s]     
\\[.3em] 
 X_i \longmapsto X_i \hspace{5.8em}
 \end{array}
\]
As an \Amoz, $\gC$ is generated by the classes  $X^\alpha \mod J=\gen{(h_{\beta})_{\beta\in\pcB}}$ (for $\alpha\in\cB$).  

Moreover $\Phi$ is surjective because $S$ has an inverse in $\Phi(\gC)$. Indeed 
  $\Phi(\gC)$ is a finitely generated $\gA$ module and the multiplication by $S$ in $\Phi(\gC)$ is
residually onto
(it is the identity), therefore
 by  Nakayama lemma the multiplication by $S$  is itself onto on $\Phi(\gC)$. 

As we get  $\Phi(\gC)=\gB[1/s]$, 
$\gB[1/s]$ is  generated as \Amo by the lifting $\so{x^\alpha;\alpha\in\cB}$ of $\cB$ and the $h_{\beta}$'s  are rewriting rules for this generator system.


\smallskip \noindent \emph{2}. We have $\gB= \AXn /\!\gen {f_1, \ldots, f_n} $. 
By item \emph{1} we know that $\gB[1/s]$ is a finitely generated \Amoz \ for some $S(X)\in 1 + \fm[\uX]$. We can apply  the theorem of de Smit \&  Lenstra \ref{thdeSmitLenstra} to conclude that $\gB[1/s]$ is a  finite free \Amoz. Clearly its rank is the same after 
$\otimes_\gA\gk$, so it is  
equal to the dimension $r=\#\cB$ of the \kev $\ov{\gB[1/ s]}=\ov\gB$.
Since $\gB[1/s]$ is generated by a lifting of $\cB$, we conclude by saying that in a  free module of rank $r$, any generator system of $r$ elements is a basis.

\medskip On the other hand, since the  
generator system $\so{x^\alpha:\alpha\in\cB}$ of the \Amo $\gC$ 
is mapped by $\Phi$ to a basis of the
\Amo $\gB[1/s]$ with the same number of elements,
we have that $\Phi$ is injective and therefore it is an isomorphism. This fact will be used in Section \ref{classical}.
In consequence we have finished the proof of the theorem.
\end{pf}

%
\medskip  
Our main concern is the following result in which item \emph{3} generalizes \cite[Theorem~11]{AL2010}.
We point out  that in \cite[Theorem~11]{AL2010}  we forgot to give as an hypothesis the fact that
$\gA$ is integrally closed in~$\gK$: the proof was given only for the  case
of a valuation domain and used implicitely the normality hypothesis in the general case.

\begin{theorem} \label{bezloc} \emph{(B\'ezout local)} Let $(\A,\fm,\gk)$ be a Henselian local ring and  
$\gB:=\Axn$ a finitely generated $\gA$-algebra. Let 
$\ov{\gB}:=\gB/\fm\gB=\gk[\ov{x_1},\dots,\ov{x_n}]$ be the residual algebra. Assume that $\ov\gB$ is finitely presented\footnote{From a constructive point of view we cannot deduce that $\ov\gB$ is finitely presented from the fact it is finitely generated.} as a \klg and that $\dim_\gk \ov\gB_{\gen{\ov{x_1},\dots,\ov{x_n}}}=r\geq 1$
(this means that $(\uze)$ is an isolated residual zero of multiplicity $r$). Then $\fm\gB+\gen{x_1,\dots,x_n }$ (denoted as $\fm+\gen{\ux}$)
 is a maximal ideal of $\gB$ and 
 we have:
 
\begin{enumerate}
\item (Pure algebraic form, without zeroes)\\ 
 The \Alg $\gB_{\fm+\gen{\ux}}$ is a finitely generated  \Amoz: 
 there exists an $u\in1+\fm\gB+\gen{\ux}$
 such that 
 $$\gB[1/u]=\gB_{\fm+\gen{\ux}},\; 
  \ov\gB[1/\ov u]=\ov\gB_{\gen{\ov{x_1},\dots,\ov{x_n}}},
 $$
  and we get a generator system of the \Amo $\gB_{\fm+\gen{\ux}}$
  by  lifting  a $\gk$-border basis of $\ov\gB[1/\ov u]$.
\item (Local complete intersection case, without zeroes)\\ If in addition $\gB$ is finitely presented as 
 \hbox{$\gB:=\AuX/\!\gen{f_1, \ldots, f_n}=\Axn$}, 
 then $\gB_{\fm+\gen{\ux}}$ is   
is a free \Amo  of rank $r$, whose basis is given by lifting any $\gk$-border basis of~$\ov {\gB}_{\gen{\ov{x_1},\dots,\ov{x_n}}} $ with its rewriting rules.
\item (Local complete intersection case, usual form, with zeroes)\\ Assume  that  $(\A,\fm,\gk)$ is 
 a local normal domain with algebraically closed quotient 
field~${\K}$ and let $\gB$ as in item 2. Then, there are finitely many zeroes of $(f_1,\ldots,f_n)$
 above the residual zero~$(\uze)$ (i.e., zeroes with coordinates 
 in $\fm$) and the sum of their multiplicities
 equals~$r$. 
\end{enumerate}
\end{theorem}

\begin{proof}

\emph{1} and \emph{2}. The hypothesis $\dim_\gk(\ov{\gB_1})=r\geq 1$ means that $(\uze)$ is residually an isolated zero of multiplicity $r$. So, we can construct an $e\in\gB$
with $\ov e\in\ov\gB$ such that
\begin{itemize}
\item $\ov e$ is an idempotent in $\ov\gB$,
\item  for a suitable integer $N$,   $\gen{1-\ov e}=\gen{\ov{x_1},\dots,\ov{x_n}}^{N}$ in $\ov\gB$,
\item  $\ov\gB[1/\ov e]=\ov{\gB_1}$.
\end{itemize}
(See \hbox{\cite[Theorem IX-4.7]{CACM}}.) 
In particular
\[ 
\begin{array}{lll} 
e\in 1+\fm\gB+\gen{\ux}  & \hbox{and}  & e^2-e\in\fm\gB   
\end{array}
\]
Let $\gC:=\gB[1/e]$. Since $e\in 1+\fm\gB+\gen{\ux}$, we have $\gC_{\fm+\gen{\ux}}=\gB_{\fm+\gen{\ux}}$, and
$$
\ov\gC=\ov\gB[1/\ov e]=\ov{\gB}_{\gen{\ov{x_1},\dots,\ov{x_n}}}=\ov{\gC}_{\gen{\ov{x_1},\dots,\ov{x_n}}}.
$$
 So $\gC$ is a finitely presented  \Alg  to which we can apply Corollary~\ref{superbezloc0} \emph{1}. There exists $s\in1+\fm\gC$ such that $\gC[1/s]=\gC_{\fm+\gen{\ux}}$.
Since $e\in1+\fm\gB+\gen{\ux}$ we see that $s$ is written as $v(x)/e^m$ for \hbox{a $v\in1+\fm\gB+\gen{\ux}$} and a suitable exponent $m$.
Finally we see that 
$$\gC[1/s]=\gB[1/u]=\gB_{\fm+\gen{\ux}}
$$ 
where $u=sv\in1+\fm\gB+\gen{\ux}$. 
This finishes the proof of item \emph{1}.

\medskip \noindent \emph{2}.
Now we give some details for describing $\gC$ and $\gC[1/s]$ when $m=n$. 
Let $E(\uX)\in\AuX$ \hbox{s.t.\ $e=E(\ux)$}, we have
\[ 
\begin{array}{lll} 
1-E\in\fm+\gen{\uX}+\gen{f_1,\dots,f_n}  & \hbox{and}  & E^{2}-E\in \fm[\uX] +\gen{f_1,\dots,f_n} \quad (\hbox{inside } \AuX)  
\end{array}
\]
Therefore  
$$
\gC:=\gB[1/e]=\A[\ux,1/e]=\gA[\uX,T]/\!\gen{\fn,E(\uX)(T+1)-1}=\A[\ux,t]
$$
Now we write $s=S(\ux,t)$ where $S(\uX,T) \in 1+\fm[\uX,T]$ 
and  Corollary~\ref{superbezloc0} \emph{2} says us that
$$\gD:=\gC[1/s] =
  \gA[\uX,T,Z]/\!\gen{\fn,E(\uX)(T+1)-1, S(\uX,T)(Z+1)-1}
$$ 
is a free \Amo of rank $r$, whose basis is given by lifting any $\gk$-border basis of~$\ov {\gB}[1/\ov e]=\ov {\gB}_{\gen{\ov{x_1},\dots,\ov{x_n}}} $ with its rewriting rules..
This finishes the proof of item \emph{2}.

\medskip\noindent  \emph{3.} First of all notice that under the hypothesis of \emph{3},
$\gA$ is a Henselian local ring and item \emph{2} applies.  Let
 $$\Sigma:=\{ f_1, \ldots, f_n,E(\uX)(T+1)-1, S(\uX,T)(Z+1)-1 \} $$ 
Since $\D$ is a free $\A$-module of rank $r$, $\D\otimes_{\A}\K$ is a finite \Kev of dimension $r$.
\\
Since $\K$ is an algebraically closed field, 
by Stickelberger's theorem (see \cite[Theorem IV-8.17]{CACM}), we can decompose $\D\otimes_{\A}\K$ as 
\begin{equation}\label{ceros}
 \prod\limits_{(\xi_{1},\dots,\xi_{n+2})\in {\mathcal Z}_{\K}(\Sigma)}
 \K[\uX, T, Z]_{(\uX-\uxi,
T-\xi_{n+1},Z-\xi_{n+2})}/\langle  \Sigma \rangle
\end{equation}
In each  $(\xi_{1},\dots,\xi_{n+2})\in {\mathcal Z}_{\K}(\Sigma)$ the corresponding local \Klg is  zero-dimensional and the sum of multiplicities is equal to $r$.

\smallskip It remains to prove that 
the points of ${\mathcal Z}_{\K}(\Sigma) $  correspond exactly
to the zeroes of $(f_1,\dots,f_n)$ in $\gK$
having their coordinates in $\fm$, and that the corresponding local \hbox{\Klgsz} are isomorphic.
\\ 
First let  $(\xi_{1},\dots,\xi_{n+2})\in {\mathcal Z}_{\K}(\Sigma)$
where $\xi_{1},\dots,\xi_{n}\in \fm$. Since $E(\ux)\in1+\fm+\gen{\ux}$, we get~$\varepsilon:=E(\xi_{1},\dots,\xi_{n})\in1+\fm$. This $\varepsilon$ has a unique inverse in $\gK$, and this inverse is written $1+\mu$ with $\mu\in\fm$.
This forces $\xi_{n+1}=\mu$.
Morevover the two local algebras at $(\xi_{1},\dots,\xi_{n})$ and
$(\xi_{1},\dots,\xi_{n+1})$ are ``equal''  
$$
\Kxn_{ \gen{x_1-\xi_{1},\dots,x_n-\xi_{n}} }\simeq\gK[x_1,\dots,
,x_{n+1}]_{ \gen{x_1-\xi_{1},\dots,x_{n+1}-\xi_{n+1}} }.
$$ 
This follows from the facts that $\gK[x_1,\dots,
,x_{n+1}]=\Kxn[1/E(\ux)]$ and that
$E(\ux)$ is invertible in $\Kxn_{(x_1-\xi_{1},\dots,x_n-\xi_{n})}$. 
Indeed 
$$
E(\xn)\equiv E(\xi_{1},\dots,\xi_{n})=\varepsilon \mod \gen{x_1-\xi_{1},\dots,x_n-\xi_{n}} \hbox{ in }\Kxn,
$$
so $\gen{E(\xn),x_1-\xi_{1},\dots,x_n-\xi_{n}}=\gen{1}$ in $\Kxn$.
\\ 
Similarly $S(\xi_{1},\dots,\xi_{n},\mu)\in1+\fm$, it has a unique inverse in $\gK$, and this inverse is written $1+\nu$ with $\nu\in\fm$.
This forces $\xi_{n+2}=\nu$  and the equality 
$$
\Kxn_{(x_1-\xi_{1},\dots,x_{n+1}-\xi_{n+1})}\simeq\gK[x_1,\dots,
,x_{n+1}]_{(x_1-\xi_{1},\dots,x_{n+2}-\xi_{n+2})}.
$$

\smallskip\noindent  
It remains to see that any zero $(\xi_{1},\dots,\xi_{n+2})\in {\mathcal Z}_{\K}(\Sigma)$ has its first $n$ coordinates in $\fm$.
Corollary~\ref{corbezloc1} (with $\gE=\gK$ and $\rho:\gA\to\gK$ the inclusion morphism) shows that $\xi_{1},\dots,\xi_{n+1}$
are in $\gA$ (recall that~$\gA$ is assumed to be integrally closed). Then Corollary~\ref{cor2bezloc1} shows that $\xi_{1},\dots,\xi_{n+1}$
are in~$\fm$.
\end{proof}

\begin{remark} \label{rembezloc} 
{\rm  Theorem \ref{bezloc} {1.} can be seen as a generalization of the Mather-Weierstrass division theorem. 
In fact let $\gk$ be a field, 
$$
\gA_n=\gk[[X_1,\dots,X_n]]_{\rm alg} 
\hbox{ (with maximal ideal }\fm_n), 
\gA_{m}=\gk[[X_1,\dots,X_n,Y_1,\dots,Y_\ell]]_{\rm alg},
$$ 
$I=\gen{F_1,\dots,F_q}$ an ideal of $\gA_{m}$, and assume that the morphism $\gA_n\to \gC=\gA_{m}/I$ is quasi-finite (i.e.\  $\gC/\fm_n\gC$ is a finite dimensional $\gk$-vector space). We will see that $\gC$ is a finite \Amoz.\\
As 
$\gA_m$ is the henselization of $(\gA_n[Y_1,\dots,Y_\ell])_{\fm_n+\gen{Y_1,\dots,Y_\ell}}$, we can assume that~$I$
is contained in some 
\[ 
\begin{array}{lll} 
(\gA_n[\uY,t])_{\fm_n+\gen{\uY,t}} &\hbox{ where }&\gA_n[\uY,t]=\gA_n[\uY,T]/\!\gen{f(\uY,T)},  \\[.3em] 
  &&f(\uY,T)\in\gA_n[\uY,T],\,f(\uze,0)\in\fm_n, \,f_T(\uze,0)\in1+\fm_n,
  \\[.3em] 
  &&t\in\gA_m \;\hbox{ and }\;  f(\uY,t)=0.
 \end{array}
\]
W.l.o.g. we assume that $F_i=P_i(\uY,t)$ where $P_i(\uY,T)\in\gA_n[\uY,T]$. We apply the theorem with $\gA=\gA_n$, 
$\gB=(\gA_n[\uY,T]/\!\gen{P_1,\dots,P_q})_{\fm+\gen{\uY}}$.
Finally we get that $\gB$ is a finite \Amoz. So, it is Henselian
and $\gC=\gB$. And $\gC$ is a finite \Amoz. This is the Mather Theorem.\\
In the case of $q=1$, we get the Weierstrass division theorem applying
\ref{bezloc} {2.}   
}\end{remark}

\section{Application}\label{classical}

This section is written in classical mathematics, allowing us to deal with $\Co$ as if it were a discrete field
(i.e. assuming we have an equality test for complex numbers).

E.g., Theorem \ref{univar-cont} is not written in a constructive form
because, since there is no zero test for elements of $\CC$, it is impossible to know in the general situation what are exactly the distinct zeroes and their multiplicities for the perturbed polynomial. 
E.g., it is a priori impossible to know 
if a monic univariate polynomial of degree two has one double root or two distinct roots. 
A constructive and continuous form of the FTA has a slightly different formulation, the best one being \cite{Ric00},
where the zeroes of a complex polynomial are seen as forming a multiset that varies continuously in a suitable complete metric space. 
See also \cite[Chapter 5]{Bi67} for another constructive formulation, and \cite[Appendix~A, page 276]{Ost73} for an optimal modulus of continuity. We think that Theorem
\ref{thBezoutLocalConcret} would need a subtle constructive reformulation, with a more precise proof
than that we give here. 

Finally we would like to point out that,  in the spirit of the proof of \cite[Eisermann]{Eis2012} which is  an analogous to 
\ref{univar-cont} for the algebraic closure $\gR[i]$ of  a real closed field $\gR$, one could prove
the classical  multivariate version  \ref{thBezoutLocalConcret} for  $\gR[i]$,  taking into account the usual manipulation of
semialgebraic continuous functions w.r.t.\ 
the topology in  $\gC^n\cong \gR^{2n}$ defined by the order of the real closed field~$\gR$.

\smallskip We start with a classical result of complex analysis, about the continuity of 
roots of a univariate 
polynomial defined over $\Co$ (the field of complex numbers),  with respect to its coefficients. 
By $\Omega(\xi, \epsilon)$ we denote the polydisc
of $\Co^n$ 
centered at 
$\xi\in \Co^n$: that is   
$$\Omega(\xi, \epsilon)=\sotq{\xi' \in \Co^n}{ \abs {\xi_j-\xi'_j} < \epsilon,\, j\in\lrbn }.
$$ 
The statement  is the following easy consequence of Rouché's Theorem.

\begin{theorem}\label{univar-cont}
 Let be  $P(Y)=Y^d+\sum_{i\in\lrb{0..d-1}} a_iY^i\in \Co[Y]$ a 
nonzero polynomial of degree $d\geq 1$, and let  $P(Y)=\prod_{i\in\lrbs}(Y-\xi_i)^{m_i}$
where $\xi_i\in \Co$, $i\in\lrbs$ are the distinct roots of $P$. Let $\epsilon >0$ be s.t. 
$\Omega(\xi_i, \epsilon) \cap \Omega(\xi_j, \epsilon)=\emptyset$ for every $i\neq j$. Then, there exists $\delta>0$ s.t. 
for every $(\wha_1, \ldots, \wha_d)\in \Omega((a_1, \ldots, a_d), \delta)$, 
the polynomial $\wh P(Y)=Y^d+\sum_{i\in\lrb{0..d-1}} \wha_iY^i$ has exactly $m_j$ roots counting with multiplicities 
in $\Omega(\xi_j, \epsilon)$, for every $j\in\lrbs$.

\end{theorem}

This section consists in giving a clear proof of the following 
theorem, which is a  classical result about 
the continuity of the points in a $0$-dimensional complete intersection 
$\Co$-algebra, as in \cite{ArBook,GH}.

\begin{theorem} \label{thBezoutLocalConcret}
Let $g_1, \ldots, g_n\in \Co[\Xn]$ be polynomials of degrees $d_1, \dots, d_n$ respectively 
and  let   $p \in \mathcal{Z}_{\Co}(g_1, \ldots, g_n)$ be an isolated zero of multiplicity $r$. 
Let  $N_{1}, \ldots, N_{n}$ be   the
number of monomials of degree respectively  $d_1, \ldots, d_n$ in the variables in $X_1, \ldots, X_n$, and let $N=\sum_iN_{d_i}$. We can see $(g_1,\dots,g_n)$ as an element $(\ua)\in\Co^{N}$. Let us consider a ``slightly perturbed system'' $(\wh{g_1},\dots,\wh{g_n})$
corresponding to an element $(\underline{\wha})\in\Co^{N}$.
Then for all $\epsilon>0$ there exists $\epsilon_1$ with $0<\epsilon_1<\epsilon$  and $\delta_1>0$ such that for all 
$(\underline{\wha})\in \Omega(\ua,\delta_1)$, the perturbed system
has only finitely many zeroes in $\Omega(p,\epsilon_1)$; moreover the sum of multiplicities of these zeroes equals $r$. 
\end{theorem}

%
\begin{proof}
 For sake of simplicity we assume 
$p=(0,\ldots,0)$, that is $r=\mathrm {dim}_{\Co}\,\Co[\ux]_{(\ux)}$ where $\Co[\ux]=\Co[\xn]=\Co[\Xn]/\!\gen {g_1, \ldots, g_n}$.
We consider  a family of new indeterminates: $$\wi a^{(i)}=(\wi a^{(i)}_j)_{j=1,\ldots, N_{i}}.$$ 
Let us denote by  $G_1(\wi a^{(1)},\uX), \ldots, G_n(\wi a^{(n)},\uX)$ the generic polynomials of degrees $d_1, \dots, d_n$
respectively.  So, $g_i=G_i(a^{(i)},\uX)$ for some $a^{(i)}\in \Co^{N_{i}}$. We set now 
$$\wi v^{(i)}_j=\wi a^{(i)}_j-a^{(i)}_j.$$ We can see the $\wi v^{(i)}_j$'s as indeterminate perturbations
of the $a^{(i)}_j$'s, letting
$\wi a^{(i)}_j=a^{(i)}_j+\wi v^{(i)}_j$.

\noindent We set  
\fbox{$\gA=\Co[[(\wi v^{(i)}_j)_{i=1,\dots,n;j=1,\dots,N_i}]]_{\rm alg}$}
  the ring of algebraic formal power series (or the ring 
\fbox{$\gA=\Co\{(\wi v^{(i)}_j)_{i=1,\dots,n;j=1,\dots,N_i} \}$} of analytic power series) 
with coefficients in~$\Co$ in the variables
$ \wi v^{(i)}_j $. In the second case this ring can be identified with 
the ring of germs of analytic functions in a 
neighborhood of $\ua=(a^{(i)}_j)_{i,j}$.
In both cases the ring is  Henselian with residue field 
$\Co$.

\smallskip \noindent Then, we apply Theorem \ref{bezloc} \emph{2} taking as $f_i$'s the $G_i$'s, 
 hence $\gB=\Aux=\AuX/\!\gen{ G_1, \ldots, G_n}$. Let us recall some elements in  the proof of  \ref{bezloc} \emph{2}.

\smallskip \noindent  We have \Algs $\gC=\gB[1/e]$ and $\D=\gC[1/s]=\gB_{\fm+\gen{\ux}}$,
and  $\D$ is a free~\Amo of rank $r$, with
$$
\ov\D=\ov\gC=\ov\gB[1/\ov e]=\ov{\gB}_{\gen{\ov{x_1},\dots,\ov{x_n}}}=\ov{\gC}_{\gen{\ov{x_1},\dots,\ov{x_n}}}=\ov{\D}_{\gen{\ov{x_1},\dots,\ov{x_n}}}.
$$

\noindent In this situation a  $\Co$-border basis $\cB$ of  $\ov{\gB_1}:=\ov{\gB}_{\gen{\ov{x_1},\dots,\ov{x_n}}}$ (with $r=\#\pcB$)
lifts with its rewriting rules to an $\gA$-basis of~$\D$. So we get
in $\D$ for every ${\beta}\in\pcB$ an equation 
\begin{equation}\label{r-w-inD}
 x^{\beta}-\sum\nolimits_{ \alpha\in \cB} U_{\beta,\alpha}  x^{\alpha} =0
\quad (\hbox{for some }U_{\beta,\alpha} \in \gA) 
\end{equation}

\noindent  These  $U_{\beta,\alpha}$  are algebraic power series in the formal coefficients $\wi v$'s and they 
 give analytic functions in a neighborhood of 
 $(\ua)=(a^{(i)}_j)_{i,j}$. 

\noindent  Let us denote by  $\Lambda_i$ the multiplication matrix by 
 $ x_i$ in the free \Amo $\D=\Span_{\gA}( \cB)$,  w.r.t.\ the basis $ \cB $ using (\ref{r-w-inD}).
 Notice that the entries of 
$\Lambda_i$ are either $0$ or $1$ or one of the 
$U_{\beta,\alpha}$'s and, since  $\cB$ is a border basis, one has  
\begin{equation}\label{conmuta}
\displaystyle \Lambda_i \Lambda_j - \Lambda_j \Lambda_i=0
\end{equation}

\vspace{-.5em}
\noindent in $\gA^{r\times r}$ for $i\neq j$. 

\noindent Now we remark that  in the proof of \ref{bezloc} \emph{2}
we have seen that the following ideals of $\gA[X_1,...,X_n]$ coincide.

\vspace{-.7em}\begin{small}
\begin{equation}\label{equalideals}
\langle X^{\beta}-\som_{\alpha\in \cB} U_{\beta,\alpha}X^{\alpha} : 
\beta\in \pcB \rangle = \langle G_1,\ldots, G_n,
(T+1)E-1, (Z+1)S-1 \rangle \gA[\uX,T,Z] \cap \gA[\uX]
\end{equation}
\end{small}


\vspace{-.7em}\noindent Hence one gets equalities $(\beta \in \pcB)$:
\begin{equation}\label{identity1}
 X^{\beta}-\som_{\alpha\in \cB}U_{\beta,\alpha}X^{\alpha} = \som_{i\in\lrbn}W_iG_i +  ((T+1)E-1)M
 + ((Z+1)S-1)R   
 \end{equation}
for some $W_i,M,R \in \gA[\uX, T, Z]$, and on the other side ($i=1,\dots,n$) 
\begin{equation}\label{identity2}\textstyle
G_i(\uX)= \sum_{\beta \in \pcB} Q_{i, \beta}(\uX)\cdot ( X^{\beta}-\sum_{ \alpha\in \cB} U_{\beta,\alpha}X^{\alpha} )
\end{equation}
for some $Q_{i, \beta} \in \gA[\uX]$. 
We know  that identities (\ref{identity1}) and 
(\ref{identity2}) hold in the smallest sub-$\Co$-algebra of $\A$ containing the $U_{\beta,\alpha}$'s.

\noindent We  recall that $\fm=\geN{\wi v^{(i)}_{j};i,j}$, 
$E(\uX)\in 1 + {\mathfrak m} + \langle \uX \rangle$ 
and  $S(\uX,T)\in 1+ {\mathfrak m}[\uX,T]$. To stress the (analytic) dependence of $E$ and $S$ 
on the $ \wi v^{(i)}_j $'s we shall write
$E(\uwv,\uX)$ and $S(\uwv ,\uX, T).$

\smallskip \noindent Now let $P_i\in \gA[Y]$ denote the characteristic polynomial of  the multiplication by the image of $x_i$ in $\D$ (i.e.\ the characteristic polynomial of $\Lambda_i$).
Its coefficients are $\ZZ$-polynomials in the  $U_{\beta,\alpha}$'s.   By change of ring,   $\ov{P_i}\in \Co[Y]$ 
is the characteristic polynomial of the multiplication by $\ov{x_i}$ in $\ov \D=\ov \gB_1$.

\noindent
From  $\epsilon$, 
 by usual arguments of Taylor calculus for convergent series  we can find $\delta$ and $\epsilon'<\epsilon$ 
 such that  the following properties hold:

\smallskip \noindent i) All the $U_{\beta,\alpha}$ are convergent and also  the  coefficients in
(\ref{identity1}) and 
(\ref{identity2}) are convergent for any $ \wha^{(i)} \in \Omega(a^{(i)}, \delta)$.
Therefore  the same will hold for the coefficients of $P_i$.
 
\smallskip \noindent ii) The polynomial $E(\uwv ,\uX)$ as a polynomial with coefficients analytic functions
in a neighborhood of $a^{(i)}$'s
verify that for every 
 $(\wha^{(i)},\ux) \in \Omega(a^{(i)}; \delta)\times \Omega(\uze ; \epsilon')$, $\lvert E(\uwha, \ux)-1\rvert < \epsilon'$.

\noindent In particular for those points is $E(\uwha, \ux) \neq 0$
 and we may ask also  $\lvert 1-\frac{1}{E(\uwha,\ux)} \rvert < \epsilon'.$
 
\smallskip  \noindent iii) In the same way for 
 $ (\wha^{(i)},\ux,t) \in \Omega(a^{(i)}; \delta)\times 
 \Omega(\uze ; \epsilon')\times \Omega(0; \epsilon')$,
  $ S(\uwha,\ux,t) $  does not vanish and $\lvert 1-\frac{1}{S(\uwha,\ux,t)}\rvert <\epsilon'$.
 
\medskip \noindent  
We apply Theorem \ref{univar-cont} with $\epsilon'$ and $P=Y^r=\ov{P_i}\in \Co[Y]$, 
 there exists $\delta'\in \RR^{+}$, $\delta'<\delta$, such that 
if $Q\in \Co[Y]$ 
is a monic polynomial of degree $r$, whose 
coefficients belong to a neighborhood $\Omega(\uze ; \delta')$, then, the distance of each of the
$r$ complex roots of $Q$
 to $0$ is  less than $\epsilon'$. Notice that $Q$ can be considered as a 
perturbation
of $\ov{P_i}$.

\noindent Then, for every $\wha^{(i)} \in \Omega(a^{(i)}; \delta')$; we consider 
  $\wh{g_i}:=G_i(\wha^{(i)},X)$, 
   which is  ``a perturbation'' of the $g_j$,
and we denote by
$\wh U_{\beta,\alpha}$  the result of specializing $U_{\beta,\alpha}$ in the values of $\wha^{(i)}$'s. 



\noindent Then  the following facts  hold true and alltogether prove the theorem taking the value $\delta_1$ as $\delta'$ 
and $\epsilon_1$ as~$\epsilon'$:

\smallskip \noindent 1.  $\{x^{\beta}: \beta\in \mathcal{B}\}$ is a basis of the $\Co$-vector space, 
because the relations 
of commutations (\ref{conmuta}) still hold under specialization at $\wha^{(i)}$ 
(see  {\bf 2.2.4} in section 2 on border bases). Consequently the ideal of $\Co[\uX]$
$\wh I=I(\wha^{(i)}):=\langle X^{\alpha}-\som_{\beta\in \cB}\wh U_{\beta,\alpha}X^{\beta}  \rangle$ has $r$ 
zeroes in $\Co^n$ counted with multiplicity.

\smallskip \noindent 2. 
The characteristic polynomial of the multiplication by $X_i$ in the ring 
$\Co[\uX]/\wh I$ is 
$\wh {P_i}$. For every~\hbox{$(\uxi) \in \Co^n$} which is a zero of $\wh I$ its coordinates $\xi_i$'s
 are roots of the polynomials $ \wh {P_i} (T)$. 
 We know that these roots belong to 
$\Omega(\uze; \epsilon')$. On the other hand these zeroes are zeroes of the ideal 
$J(\wha^{(i)}):=  \langle (\wh {g_i})_{ i=1,\dots,n}  \rangle\Co[\uX]$,
because of the specialization of (\ref{identity2}).

\smallskip \noindent 3. Reciprocally given   $(\uxi)  \in \Co^n$ which is a zero of the ideal $J(\wha ^{(i)})$ and such that
$(\uxi)  \in \Omega(\uze; \epsilon') $ we are going to prove that $(\uxi) $ is a zero of $I(\wha^{(i)})$.
As $\wha^{(i)} \in $ and $(\uxi)  \in \Omega(\uze; \epsilon') $, by ii) above we have that 
$t:=1-\frac{1}{E(\wha^{(i)},\uxi)}\in \Omega(0,\epsilon')$. 
Consequently it makes sense to substitute in (\ref{identity1}),   
$(\wha^{(i)} , \uxi , t)$. since the 
right hand side vanishes, the same happen for the left hand side.

\smallskip \noindent 4. Given $(\uxi) \in  \Omega(\uze; \epsilon') $,  zero of $J(\wha^{(i)})$  and hence of $I(\wha^{(i)})$
the local rings $\Co[\uX]/I(\wha^{(i)})_{(\uX-\uxi )} $ and $\Co[\uX]/J(\wha^{(i)})_{(\uX-\uxi )}  $ coincide.
In fact this comes from (\ref{identity2}), and by substituting in (\ref{identity1}) $T$ by $1-\frac1{E(\wha^{(i)},\uX)}$
 and  $Z$ by $1-\frac{1}{S(\wha^{(i)},\uX,T)}$, which is allowed since both expressions 
 are rational regular at the point~$(\uxi)$.
\end{proof}

\bibliographystyle{plain}
\bibliography{AlonsoLombardi2015}

\end{document}